\documentclass[10pt,reqno,a4paper]{amsart}

\usepackage[pdftex]{graphicx}
\usepackage{pgfplots}
\usepackage{tikz}  
\usepackage{caption}

\pgfplotsset{compat=1.18}

\usepackage{amsmath}
\usepackage{amssymb,amsthm,hyperref,caption}
\usepackage{xcolor}
\definecolor{meinBlau}{rgb}{0.2,0.2,0.9} 
\definecolor{blau}{rgb}{0,0,0.75} 
\definecolor{rot}{rgb}{0.74,0,0} 
\definecolor{darkgreen}{rgb}{0,0.4,0.1} 
\definecolor{codegreen}{rgb}{0,0.6,0} 
\definecolor{codeblue}{rgb}{0,0.1,0.7} 

\definecolor{myblauA}{rgb}{0,0.1,1} 
\definecolor{myblauB}{rgb}{0,0.2,0.8} 
\definecolor{myblauC}{rgb}{0,0.3,0.7}
\definecolor{myblauD}{rgb}{0,0.4,0.55}
\definecolor{myblauE}{rgb}{0,0.5,0.4}
\definecolor{myblauF}{rgb}{0,0.6,0.3}

\hypersetup{colorlinks,linkcolor=blau,citecolor=blue,urlcolor=meinBlau}
\usepackage{times}
\usepackage{enumitem}
\usepackage{verbatim}
\usepackage{listings}
\usepackage{fancyvrb}

\lstdefinestyle{Rstyle}{
language=R, 
backgroundcolor=\color{white}, 
commentstyle=\color{codegreen}, 
keywordstyle=\color{codeblue}, 
numberstyle=\color{codeblue}, 
stringstyle=\color{codegreen}, 
basicstyle=\ttfamily\tiny, 
morekeywords={TRUE,FALSE}, 
deletekeywords={data,frame,length,as,character}, 
showstringspaces=false 
}

\allowdisplaybreaks

\newtheorem{theorem}{Theorem}
\newtheorem{lem}[theorem]{Lemma}

\newtheorem{prop}{Proposition}

\theoremstyle{definition}
\newtheorem{code}{Code}
 
\newtheorem{remark}{Remark}
\newtheorem{example}{Example}
\newtheorem{defi}{Definition}

\def\P{{\mathbb {P}}}
\def\E{{\mathbb {E}}}

\newcommand{\fallfak}[2]{\ensuremath{#1^{\underline{#2}}}}
\newcommand{\auffak}[2]{\ensuremath{#1^{\overline{#2}}}}

\newcommand{\N}{\ensuremath{\mathbb{N}}}

\newcommand{\R}{\ensuremath{\mathbb{R}}}

\DeclareMathOperator{\fc}{FC}
\newcommand{\FC}{\ensuremath{\fc_n}}
\newcommand{\FCn}[1]{\ensuremath{\fc_{#1}}}
\DeclareMathOperator{\Cut}{Cut}
\DeclareMathOperator{\Rif}{Rif}
\DeclareMathOperator{\id}{id}

\newcommand{\gro}{\ensuremath{\mathcal{O}}}

\newcommand{\pe}{\ensuremath{p^{\ast}}}

\DeclareMathOperator{\law}{\overset{\mathcal{L}}{=}}
\DeclareMathOperator{\claw}{\overset{\mathcal{L}}{\rightarrow}}

\DeclareMathOperator{\Erl}{Erl}

\DeclareMathOperator{\Exp}{Exp}

\DeclareMathOperator{\Bin}{B}
\DeclareMathOperator{\Po}{Po}
\DeclareMathOperator{\Be}{Be}

\begin{document}

\author[M.~Kuba]{Markus Kuba}
\address{Markus Kuba\\
Department Applied Mathematics and Physics\\
University of Applied Sciences - Technikum Wien\\
H\"ochst\"adtplatz 5, 1200 Wien} %
\email{kuba@technikum-wien.at}

\title[Card guessing after an asymmetric riffle shuffle]{Card guessing after an asymmetric riffle shuffle}


\begin{abstract}
We consider a card guessing game with complete feedback. An ordered deck of
$n$ cards labeled $1$ up to $n$ is riffle-shuffled exactly one time.
Given a value $p\in(0{,}1)\setminus\{\frac12\}$, the riffle shuffle is assumed to be unbalanced, such that the cut is expected to happen at position $p\cdot n$. The goal of the game is to maximize the number of correct guesses of the cards: one after another a single card is drawn from the top, and shown to the guesser until no cards remain. We provide a detailed analysis of the optimal guessing strategy and study the distribution of the number of correct guesses. 
\end{abstract}

\maketitle

\section{Introduction}
\subsection{Riffle shuffle and card guessing} In this work we study the Gilbert–Shannon–Reeds model for riffle shuffling $n$ cards, $n\in\N$, sometimes also known as dovetail-shuffling. 
Informally speaking, an ordered deck of cards labeled one up to $n$ is cut near the middle into
two piles according to a binomial distribution with parameter one-half. Then, the cards are riffled together according
to the following rule: if the left packet has $k$ cards and the right has $n-k$ cards, drop the next
card from the left packet with probability $k/n$ and from the right packet
with probability $(n-k)/n$, with $0\le k\le n$. Then, we continue until all cards have been dropped. 
A great many questions are connected to this shuffling procedure. We mention the classical question of randomization of an ordered deck of cards, where we point out the works of Aldous~\cite{Aldous1983}, Aldous and Diaconis~\cite{AldousDiaconis1986}, Bayer and Diaconis~\cite{BayerDiaconis} and refer the reader to the recent work of Assaf, Diaconis and Soundararajan~\cite{ADS2011} for a thorough review of the literature regarding this question. 
The topic of this work is the analysis of a card guessing procedure. Given a once riffle-shuffled deck, a person guesses the cards and tries to maximize the number $X_n$ of correct guesses. One may consider different feedback rules like complete feedback, the drawn card is shown to the guessing person, or no feedback at all, such that the identities of the cards are not revealed, nor is the guesser told whether a particular guess was correct or not. We emphasize that such questions are also classical and refer the reader to the work of Ciucu~\cite{Ciucu1998}, where he studied an optimal strategy under the no feedback rule. Progress on this was recently obtained in~\cite{KSTY2023,NFNW-KT2022}. Concerning the complete feedback variant, Liu~\cite{Liu2021} and also Krityakierne and Thanatipanonda~\cite{KT2023} made progress on the guessing problem assuming that the guesser is using the optimal strategy. An asymptotic expansion of the expected value $\E(X_n)$ is provided for $n$ tending to infinity in~\cite{Liu2021}, whereas  an enumerative analysis and a study of higher moments has been carried out in~\cite{KT2023}. Therein, also precise asymptotics of the first few moments $\E(X_n)$, $\E(X_n^2)$, etc. were obtained using both enumerative and symbolic methods. Building on the earlier results, limit laws have been obtained for no feedback in~\cite{KuPa2024}, and for complete feedback in~\cite{KuPa2025}. In the classical setting of the Gilbert–Shannon–Reeds model, the cut of the ordered deck happens at the middle of the deck. The aim of this work is to study an asymmetric version of the Gilbert–Shannon–Reeds model in the context of the card guessing game with full feedback. First, we discuss the asymmetric Gilbert-Shannon-Reeds model, where the cut is expected to happen at position $n\cdot p$, for a parameter $p\in [0,1]$, such that the two packets sizes are expected to be unbalanced. For $p=\frac12$ we reobtain the ordinary symmetric variant.
Such shuffles are sometimes called $p$-shuffles or biased riffle shuffles. They were introduced in \cite{DFP92} and further studied in~\cite{Ful98,Lal96,Lal00}, as well as~\cite{Sellke}.
We note that the boundary values $p=0$ or $p=1$ lead to a degenerate ordered deck with probability one. Our aim is to study the properties of the asymmetric model and determine the optimal guessing strategy for card guessing with full feedback, depending on the value of $p$. We determine the distribution of the number of correct guesses $X_n$, starting with $n$ cards labeled one up to $n$, once asymmetrically riffle shuffled, via a recursive equation and derive limit laws.

\subsection{Card guessing with colored cards}
Another card guessing games has been considered in the literature in great many articles~\cite{Diaconis1978,DiaconisGraham1981,HeOttolini2021,KnoPro2001,KP2024,KuPanPro2009,Leva1988,OttoliniSteiner2023,OT2024,Read1962,Zagier1990}. A deck of a total of $M$ cards is shuffled, 
and then the guesser is provided with the total number of cards $M$, as well as with the individual numbers of types, say hearts, diamonds, clubs and spades. After each guess of the type of the next card, the person guessing the cards is shown the drawn card, which is then removed from the deck. This process is continued until no more cards are left. Under the assumption that person guessing the cards tries to maximize the number of correct guesses, one is interested in the total number of correct guesses. 
In the simplest non-trivial setting there are only two colors, red (grouping together hearts and diamonds) and black (clubs and spades).
Their numbers are given by non-negative integers $m_1$, $m_2$, with $M=m_1+m_2$. One is then interested in the random variable $C_{m_1,m_2}$, counting the number of correct guesses. Here, we remark that not only the distribution and the expected value of the number of correct guesses is known~\cite{DiaconisGraham1981,KnoPro2001,Leva1988,Read1962,Zagier1990}, but also multivariate limit laws and additionally interesting relations to combinatorial objects such as Dyck paths and urn models are given~\cite{DiaconisGraham1981,KP2024,KuPanPro2009}. For the general setting of $n$ different types of cards we refer the reader to~\cite{DiaconisGraham1981,HeOttolini2021,OttoliniSteiner2023,OT2024} for recent developments. It turned out that this guessing game is also intimidatingly related to the riffle shuffling model, as the analysis of the riffle-shuffled guessing procedure is linked to the guessing game with different colors~\cite{KuPa2025}. 
The different card guessing games mentioned in the introduction are not only of purely mathematical interest. There are applications to the analysis of clinical trials~\cite{BlackwellHodges1957,Efron1971}, fraud detection related to extra-sensory perceptions~\cite{Diaconis1978}, guessing so-called Zener Cards~\cite{OttoliniSteiner2023}, as well as relations to tea tasting and the design of statistical experiments~\cite{Fisher1936,OT2024}.

\subsection{Main results}
The main results of this work are the optimal strategy for the complete feedback model, as well as the limit laws for $X_n$, for $0<p<1$, with $p\neq \frac12$. For the reader's convenience we collect here the main distributional result of Theorem~\ref{the:1}, together with the previous result for the case $p=\frac12$~\cite{KuPa2025}, as well as the boundary values $p\in\{0,1\}$. Note that we write $X_n\sim a_n+b_n \cdot X$, if $(X_n-a_n)/b_n\claw X$. 

\begin{theorem}\label{the:Summary}
Given $p\in[0,1]$, let $\pe=\max\{p,1-p\}$. The random variable $X_n$, counting the number of correct guessing with full feedback after an (a)symmetric riffle shuffle with parameter $p$, satisfies:
\[
X_n \sim  n\cdot p^{\ast} +
\begin{cases}
D, \quad\text{for } p\in\{0,1\},\\
G, \quad\text{for } p\in(0,1)\setminus\{\frac12\},\\
\sqrt{n}\cdot GG,\quad\text{for } p= \frac12,
\end{cases}
\]
with $D$ denoting a degenerated random variable, such that $\P\{D=0\}=1$, $G$ denoting a geometric distribution with parameter $\rho=(1-\pe)/\pe$, $\P\{G=k\}=\rho^k(1-\rho)$, $k\ge 0$ and 
$GG$ a generalized gamma distributed random variable, with density $f(x)=\sqrt{\frac{2}{\pi}}\cdot 8x^2e^{-2x^2}$, $x\ge 0$.
\end{theorem}
The generalized gamma random variable is also known as a Maxwell-Boltzmann distribution. This theorem indicates to look more closely at the values $p=1$ and $p=\frac12$, where the behavior of $X_n$ changes. Thus, we additionally 
focus in detail on the transitions for $p\to 1$ and $p\to\frac12$, with $p=p(n)$, and also briefly comment on the symmetric behavior $p\to 0$. 
These transition can be observed assuming that $p=\frac12 +\alpha_n$, with $\alpha_n\to 0$. and also $p=1-\alpha_n$, with $\alpha_n>0$ and $\alpha_n\to 0$. Depending on how fast $\alpha_n\to 0$ for $n$ tending to infinity, we will observe several phase transitions  the limit law of $X_n$ in Theorems~\ref{the:Half} and~\ref{the:PoN}, our two other main results beyond Theorem~\ref{the:Summary}.

\subsection{Notation and plan of the paper}
In the next section, we present the asymmetric Gilbert-Shannon-Read model and record three equivalent descriptions of it. We also determine the distribution of the value of the first drawn card, as well as the optimal guessing strategy with complete feedback model. This leads then to a distributional equation for the number $X_n$ of correct guesses in Theorem~\ref{the:distDecomp}.
Section~\ref{sec:LimitLaw} is devoted to the derivation of the limit law of $X_n$. In Section~\ref{sec:Boundary} we look in detail into phase transitions for the parameter $p$.

As a remark concerning notation used throughout this work, we always write $X \law Y$ to express equality in distribution of two random variables (r.v.) $X$ and $Y$, and $X_{n} \claw X$ for the weak convergence (i.e., convergence in distribution) of a sequence
of random variables $X_{n}$ to a r.v.\ $X$. Furthermore we use $\fallfak{x}{s}:=x(x-1)\dots(x-(s-1))$ for the falling factorials, and $\auffak{x}{s}:=x(x+1)\dots(x+s-1)$ for the rising factorials, $s\in\N_0$. Moreover, $f_{n} \ll g_{n}$ denotes that a sequence $f_{n}$ is asymptotically smaller than a sequence $g_{n}$, i.e., $f_{n} = o(g_{n})$, $n \to \infty$, and $f_{n} \gg g_{n}$ the opposite relation, 
$g_{n} = o(f_{n})$.

\section{Asymmetric Gilbert-Shannon-Read model}
\subsection{Riffle shuffle model~\label{Subsection_GSR}}
A riffle shuffle is a certain card shuffling technique. In the mathematical modeling of card shuffling, the \emph{Gilbert–Shannon–Reeds} model~\cite{BayerDiaconis,DiaconisGraham1981,Gilbert1955} describes a probability distribution for the outcome of such a shuffling. 
We consider a sorted deck of $n$ cards labeled consecutively from 1 up to $n$. 
The deck of cards is cut into two packets, assuming that
the probability of selecting $k$ cards in the first packet and $n-k$ in the second packet is defined as a binomial distribution with parameters $n$ and $1/2$, such that $\P\{\Cut=k\}=\frac{\binom{n}k}{2^n}$, $0\le k\le n$, and the cut is expected to happen at position $n/2$ and the two packets nearly balanced. Afterward, the two packets are interleaved back into a single pile. In this work we consider an asymmetric version. Given $p\in(0,1)$, we assume that the distribution of the cutting position follows a binomial distribution $\text{Bin}(n,p)$:
\begin{equation*}
\P\{\Cut=k\}=\binom{n}k p^k (1-p)^{n-k}, \quad 0\le k\le n.
\end{equation*}
Then, as usual, one card at a time 
is moved from the bottom of one of the packets to the top of the shuffled deck, such that if $m_1$ cards remain in the first and $m_2$ cards remain in the second packet, then the probability of choosing a card from the first packet is $m_1/ ( m_1 + m_2 )$ and the probability of choosing a card from the second packet is $m_2 / ( m_1 + m_2 )$. See Figure~\ref{fig:RiffleMerge} for an example of a riffle shuffle of a deck of five cards.
\begin{figure}[!htb]
\includegraphics[scale=0.55]{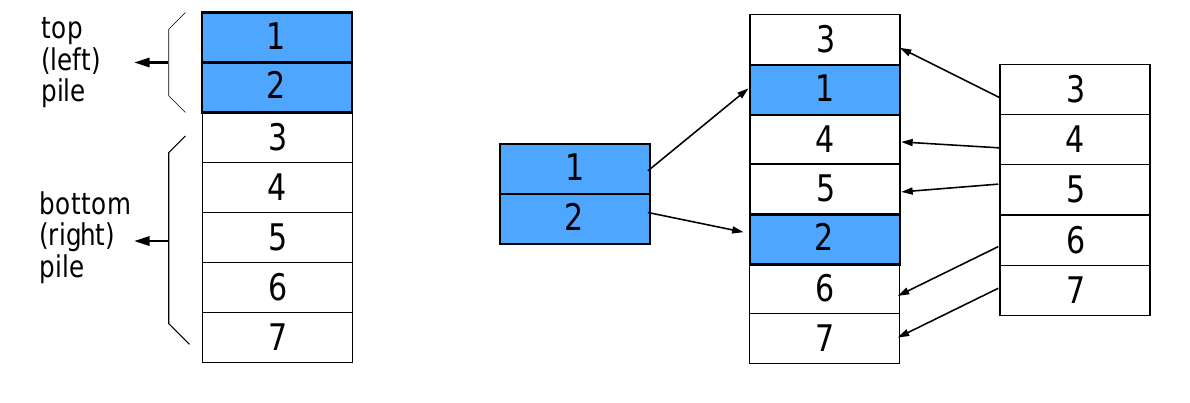}%
\caption{Example of a one-time riffle shuffle: a deck of seven cards is split after the number two with probability $\binom72\cdot p^2(1-p)^5$ and then interleaved.}%
\label{fig:RiffleMerge}%
\end{figure}

The asymmetric riffle shuffle defines a measure, which we denote by $\Rif_p$, on the symmetric group $\mathcal{S}_n$.
\begin{defi}
\label{eqn:defRif}
Given a parameter $p\in(0,1)$. The measure $\Rif_p$, on the symmetric group $\mathcal{S}_n$ is defined by
$\Rif_p\colon \mathcal{S}_n\to [0,1]$,
\begin{equation*}
\Rif_p(\sigma)=
\begin{cases}
\,\displaystyle{\sum_{k=0}^{n}p^k(1-p)^{n-k}
=\frac{(1-p)^{n+1}-p^{n+1}}{1-2p}},\quad \sigma=\id,\\[0.45cm]
\displaystyle{p^k(1-p)^{n-k}},\quad \sigma\text{ has two increasing subsequences: }1,\dots,k\\
\hspace{2.5cm}\text{and } k+1,\dots,n \text{ with } 1 \le k \le n-1\\[0.1cm]
\,0,\quad \text{otherwise}.
\end{cases}
\end{equation*}
\end{defi}
The ordinary Gilbert-Shannon-Read model is reobtained by setting $p=\frac12$. 
The outcome of an asymmetric one-time shuffling in this model can also be generated from $\{a,b\}$-sequences of length $n$. 
A random word $W$ of length $n$ over the alphabet $\{a,b\}$ is generated by $n$ independent identically distributed Bernoulli trials
\begin{equation}
\label{eqn:randomWord}
W=W_1W_2\dots W_n,
\end{equation}
with 
\begin{equation}
\label{eqn:randomWord2}
\P\{W_k=a\}=p,\quad \P\{W_k=b\}=1-p,\qquad 1\le k\le n.
\end{equation}
 It is beneficial to collect all different possible descriptions of the riffle shuffle.
\begin{lem}
\label{lem:description}
The asymmetric riffle shuffle with parameter $p$ of an ordered deck of cards can be described in three equivalent ways:
\begin{enumerate}
	\item The measure $\Rif_p$ on the symmetric group $\mathcal{S}_n$ in Definition~\ref{eqn:defRif};
	\item Performing an asymmetric cut according to a binomial distribution $B(n,p)$, followed by a random interleaving of the resulting two packs, as described in Subsection~\ref{Subsection_GSR};
	\item A random word $W=W_1W_2\dots W_n$ over the letters $\{a,b\}$, where the random variables $W_k$ are iid  random variables~\eqref{eqn:randomWord2}.
\end{enumerate}
\end{lem}
\begin{proof}
First, we notice the equivalence between the random word and the outcome of the riffle shuffle. 
Assume that the word has $k$ letters $a$, which happens with probability $p^k(1-p)^{n-k}$. By replacing the letters $a$ by the increasing sequence $1, 2, \dots, k$ and the letters $b$ by the increasing sequence $k+1, k+2, \dots, n$, we obtain an asymmetric riffle-shuffled pack,
where the cut has happened at position $k$, as there are $\binom{n}{k}$ different ways to arrange the letters $a$ and $b$. 
On the other hand, after a cut at position $k$, there are $1/\binom{n}{k}$ different outcomes, each corresponding to a different arrangement of the letters $a,b$. Finally, we note that the identity $\id\in\mathcal{S}_n$ occurs
when after a cut when the two packs are arranged just in the right order, or equivalently, when 
all of the $k$ letters $a$ appear before the $n-k$ letters $b$:
\[
\P\{\id\}=\sum_{k=0}^{n}\P\{\Cut=k\}\cdot \frac1{\binom{n}k}=\sum_{k=0}^{n}p^k(1-p)^{n-k}.
\]
Moreover, two increasing subsequences, one of the form $1,2\dots, k$, occur
with probability $p^k(1-p)^{n-k}$, as stated in Definition~\ref{eqn:defRif}.
\end{proof}

\subsection{First drawn card\label{SubSec:FC}}
In the following we will obtain the optimal strategy after a one-time asymmetric riffle shuffle for maximizing the number $X_n$ of correctly guessed cards, starting with a deck of $n$ cards, based on asymmetric the Gilbert-Shannon-Reads model. The ordinary case $p=1/2$ was discussed before in the literature~\cite{KT2023,Liu2021}; see also the earlier work~\cite{DiaconisGraham1981,Gilbert1955}. In order to do so, we first need to determine which card appears with the highest probability at the first draw.
\begin{lem}[Distribution of the first drawn card]
\label{lem:firstCard}
Assume that a deck of $n$ cards has been riffle shuffled once in the asymmetric Gilbert-Shannon-Reads model with parameter $p\in(0,1)$. 
The probability $\P\{\FC=m\}$ of the first card being $m$, $1\le m \le n$, 
is given by
\begin{equation}
\P\{\FC=m\} = 
\begin{cases}
\displaystyle{p +(1-p)^n},\quad \text{for }m=1,\\[0.2cm]
\displaystyle{\binom{n-1}{m-1}p^{m-1}(1-p)^{n-m+1}},\quad \text{for } 2\le m\le n.
\end{cases}
\end{equation}
\end{lem}
\begin{remark}
We note that that for $m\ge 2$ we have
\[
\P\{\FC=m\}=(1-p)\P\{B+1=m\},
\]
where $B=B(n-1,p)$ denotes a binomial distribution.
\end{remark}
\begin{proof}
For $m=1$ we condition on the position of the cut and obtain
\begin{align*}
\P\{\FC=1\} &= \sum_{k=0}^{n}\P\{\Cut=k, 1 \text{ on top}\}\\
&= (1-p)^n+\sum_{k=0}^{n}\binom{n}{k}p^k(1-p)^{n-k}\cdot\binom{n-1}{k-1}\frac{(n-k)!\fallfak{k}{k-1}}{n!},
\end{align*}
where the sum easily simplifies to $p$ by the binomial theorem. On the other hand, we note that all random words of length $n$ with a leading $a$, 
and also the case of $n$ letters $b$, $W=b\dots b$, also has a leading one. For $m\ge 2$ the cut has to occur at position $m-1$
and we select all the other cards before last choice the card $m$, which can happen in $\binom{n-1}{m-1}$ different ways. 
Equivalently, we can also argue that the random word has $m-1$ letters $a$, $n-m+1$ letters $b$, 
with a leading $b$. 
\end{proof}
\begin{example}
\label{Ex:1}
We continue the example from Figure~\ref{fig:RiffleMerge}, $n=7$, with a concrete value $p=0.3$. 
The cut is expected to occur at $np=7/3$, and happens at two. Although $p<0.5$, we still guess the number one as the optimal value, 
as 
\[
p+(1-p)^n=0.3+0.7^7\approx 0.382 
\]
is bigger than the other values, see the table below.
\begin{table}[!htb]
\begin{tabular}{|c|c|c|c|c|c|c|c|}
\hline
m&1&2&3&4&5&6&7\\
\hline
$\P\{\FCn{7}=m\}$&0.382&0.212&0.227&0.13&0.042&0.007&0.005\\
\hline
\end{tabular}
\caption{Concrete values for the probabilities of the first card}
\end{table}
\end{example}

\begin{example}
\label{Ex:2}
We consider the case $n=20$ and $p=0.1$. 
The cut is expected to occur at $n p=2$. We have $\P\{\FCn{20}=1\}=0.1+0.9^{20}\approx 0.2256 $. We either guess two or three for the first card:
\[
\P\{\FCn{20}=2\}=\P\{\FCn{20}=3\}=\max_{2\le m\le 20}\P\{\FCn{20}=m\}\approx 0.2567.
\]

\end{example}

\subsection{Optimal strategy:\label{SubSec:Optimal}}
Now we turn to the \emph{optimal strategy}. Assume that a deck of $n$ cards has been riffle shuffled once in the asymmetric Gilbert-Shannon-Reads model with parameter $p\in(0,1)$. 
The strategy depends on the value of $p\in(0,1)$, as it turns out that the standard value $p=1/2$ is important in the sense that the strategy slightly changes for $p<1/2$ versus $p\ge 1/2$.

\begin{itemize}
	\item Case $\frac12 \le p <1$. The guesser should guess 1 on the first
card, as his chance of success is more than $50\%$ by Lemma~\ref{lem:firstCard}.
If the first guess was correct, then the person continues with guessing the number two, etc., i.e., as long as all previous such predictions turned out to be correct, the guesser makes a guess of the number $j$ for the $j$-th card. If such a prediction turns out to be wrong, i.e., gives a number $m>j$ for the $j$-th card, then we immediately can determine the two involved remaining subsequences $j, j+1, \dots, m-1$ and $m+1, \dots, n$, and all the numbers of the remaining cards are again guessed according to the proportions of the lengths of the remaining subsequences until no cards are left.

\item Case $0<p< \frac12$. First, we observe a transition of the first card, $\P\{\texttt{FC}=1\}$, at $n_0=\left\lfloor \frac{\ln(1/2-p)}{\ln(1-p)}\right\rfloor$ :
\[
\P\{\FC=1\}=p +(1-p)^n
\begin{cases}
<\frac12,\quad n>n_0\\
\ge \frac12 \quad n\le n_0.
\end{cases}
\]
This implies that we always compare first the values of $n$ and $n_0$. We start with the subcase of small $n\le n_0$, such that $p +(1-p)^n\ge 1/2$. Here, we continue exactly as in the case $p\ge \frac12$. For the subcase of intermediate $n$, such that $p +(1-p)^n<1/2$, the first guess depends on whether the maximum of $\binom{n-1}{m-1}p^{m-1}(1-p)^{n-m+1}$, $2\le m\le n$, basically the mode of a binomial distribution $B(n-1,p)$ with an additional factor $1-p$, is larger than $p +(1-p)^n$. If it is, we guess the shifted position $\kappa_n=1+\lfloor n\cdot p \rfloor$ of the mode, and if not, we guess the number one. 
If the guess of $\kappa_n$ was correct, we know again the proportions of the lengths of the remaining subsequences and guess accordingly. 
Similarly, if we have observed any number $m$, with $2\le m\le n$, then, we may immediately determine the two involved remaining subsequences $1, 2, \dots, m-1$ and $m,\dots, n$. All the numbers of the remaining cards are again guessed according to the proportions of the lengths of the remaining subsequences until no cards are left. On the other hand, if we observe the number one, then we continue with $n$ replaced by $n-1$, and the a deck of size $n-1$, labeled $2,3,\dots n$. 
Finally, in the subcase of large $n$ we have $p +(1-p)^n\approx p$, as $(1-p)^n$ decays exponentially. The maximum of $\binom{n-1}{m-1}p^{m-1}(1-p)^{n-m+1}$, $2\le m\le n$ tends to zero proportionally to $1/\sqrt{n}$ by the local limit theorem for the binomial distribution. Let $n_1$ denote the minimal $n\ge n_0$ such that $\P\{\FC=1\}\ge \max_{2\le m\le n}\P\{\FC=m\}$. Then, we guess again the number one for all large $n>n_1$. Let 
\begin{equation}
\label{eq:defMode}
\kappa_n=1+\lfloor np \rfloor=1+\text{mode}\big(B(n-1,p)\big).
\end{equation}
Moreover, let $A_n$ denote the non-random event 
\[
\Big\{\binom{n-1}{\kappa_n}p^{\kappa_n}(1-p)^{n-\kappa_n}<p +(1-p)^n\Big\}, 
\]
and $\mathbb{I}(A_n)$ its indicator.
We summarize the ranges and the optimal guessing strategy below.
\begin{table}[!htb]
\begin{tabular}{|c|c|c|c|}
\hline
Range & $1\le n\le n_0$ & $n_0<n\le n_1$ & $n_1<n$\\
\hline
Guess & $1$ & $\mathbb{I}(A_n)+(1-\mathbb{I}(A_n))\cdot\kappa_n$  & $1$ \\
\hline
\end{tabular}
\caption{Overview: optimal strategy for the first guess with $0<p<\frac12$}
\end{table}
\end{itemize}

\begin{prop}[Optimal strategy with an asymmetric Gilbert-Shannon-Read model]
\label{Prop:Optimal}
The strategy outlined in Subsection~\ref{SubSec:Optimal} is optimal, but not unique.
\end{prop}
\begin{remark}
Of course, it may happen that we always guess $1$, even for $0<p<\frac12$. We have not tried to determine if there is a threshold $0<p_t<\frac12$ 
such that for $p<p_t$ we only guess one and for $p_t<p$ we have $n_0<n_1$, such that we guess at least once a different value.
\end{remark}

\begin{proof}
The recursive nature of the guessing strategies is implied by Lemma~\ref{lem:description} 
and the construction by a random word, as the outcome of the first letter $W_1$ of $W=W_1W_2\dots W_n$ is independent of the remaining subword $W_2\dots W_n$.
Thus, after having observing a one, the subword $W_2\dots W_n$ corresponds to a random word over $\{2,\dots,n\}$ of length $n-1$.
Lemma~\ref{lem:firstCard} provides the probabilities of the first card observed. We already note that for $B=B(n-1,p)$ we have
\[
\P\{\FC=m\}=\binom{n-1}{m-1}p^{m-1}(1-p)^{n-m+1}=(1-p)\P\{B=m-1\},
\]
$2\le m\le n$. Also, the well known mode $\lfloor n p\rfloor$~\eqref{eq:defMode} of a binomial distribution $B(n-1,p)$ leads to the maximal value of $\P\{\FC=m\}$, $2\le m\le n$. This implies the optimality, as well always guess the card with the highest probability. Also, we observe that the optimal strategy may not be unique, as the mode may be obtained at two different values by standard properties of the binomial distribution for $n \cdot p\in \N$; compare with Example~\ref{Ex:2}. Moreover, it may also happen that $\P\{\FC=1\}=\max_{2\le m\le n}\P\{\FC=m\}$, adding to the non-uniquenes of the optimal strategy.
\end{proof}

\begin{example}
\label{Ex:3}
We consider the case $p=0.15$. Here, for $n\le n_0=\left\lfloor \frac{\ln(1/2-p)}{\ln(1-p)}\right\rfloor=6$ we have $\P\{\texttt{FC}=1\}\ge 0.5$.
For $n>n_0=6$ we have $\P\{\texttt{FC}=1\}< 0.5$. It is decreasing exponentially, whereas $\max_{2\le m\le n}\P\{\FC=m\}$ decreases proportionally to $1/\sqrt{n}$ by the local limit theorem for the binomial distribution. We guess one for $1\le n\le 14$, then we guess $\kappa_n=\lfloor n\cdot 0.15 \rfloor$ for $15\le n\le 39$, and again one for $n\ge 40$.
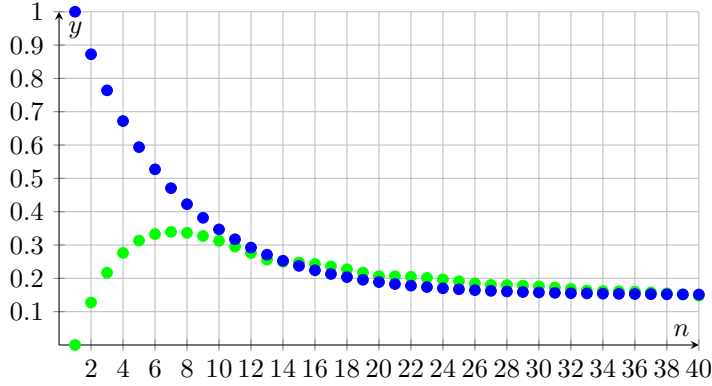
\begin{figure}[!htb]
   \begin{tikzpicture}
   \begin{axis}[
       width=10cm,
       height=6cm,
      xmin=0, xmax=40,
      ymin=0, ymax=1,
			xtick={0,2,...,40}, 
			ytick={0,0.1,...,1}, 
      xlabel={$n$},
      ylabel={$y$},
			grid=both,
      axis y line=center,
      axis x line=middle,
      ]
				 \addplot[green,only marks] table [x=x, y=y, col sep=comma] {unRifflePoints.csv};
				\addplot[samples at={0,...,40},only marks, mark size=2, blue]{0.15+0.85^x};
   \end{axis}
   \end{tikzpicture}
	\caption{Comparison of the values $\P\{\FC=1\}=0.15+0.85^n$ and $\max_{2\le m\le n}\P\{\FC=m\}$, $1\le n\le 40$.}
	\end{figure}
\end{example}

\subsection{Distributional equation}
\begin{theorem}[One-time asymmetric riffle shuffle: distribution]
\label{the:distDecomp}
Let $\kappa_n=1+\lfloor n p \rfloor$ and let $A_n$ denote the non-random event 
\[
\Big\{\binom{n-1}{\kappa_n}p^{\kappa_n}(1-p)^{n-\kappa_n}<p +(1-p)^n\Big\}, 
\]
and $\mathbb{I}(A_n)$ its indicator. The random variable $X=X_n$ of correctly guessed cards, starting with a deck of $n$ cards, after a one-time asymmetric riffle shuffle with parameter $p\in(0,1)$ satisfies for $n\ge 2$ the following distributional equation:
\begin{equation}\label{eqn:Xn_DistEqn}
\begin{split}
X_n &\law I_1 \big(X^{\ast}_{n-1}+\mathbb{I}(A)\big)+(1-I_1)I_2\cdot n \\
&\,\,+ (1-I_1)(1-I_2)\cdot \Big(C_{n-1-J_n,J_n}+(1-\mathbb{I}(A_n))\mathbb{I}(J_n=n-1-\kappa_n)\Big).
\end{split}
\end{equation}
The random variables appearing in the decomposition are defined as follows: $I_1\law \Be(p)$, $I_2\law \Be((1-p)^{n-1})$ and $C_{m_1,m_2}$ denotes the number of correct guesses in a two-color card guessing game, with $\Be(q)$ denoting a Bernoulli distribution, such that $P\{I=1\}=q$ and $P\{I=0\}=1-q$ for $I\law \Be(q)$. Additionally, $X^{\ast}_{n-1}$ is an independent copy of $X$ defined on $n-1$ cards. 
Moreover, $J_n\law \Bin_{\text{pos}}(n-1,q)$ denotes a binomial distribution, conditioned to be positive:
\[
\P\{J_n=j\}=\binom{n-1}{j}p^j(1-p)^{n-1-j}/(1-(1-p)^{n-1}),\quad  1\le j \le n-1.
\]
All random variables $I_1$, $I_2$, $J_n$, as well as $C_{m_1,m_2}$ are mutually independent. The initial value is $X_1=1$.
\end{theorem}
\begin{remark}
As pointed out before, in the case $p\ge \frac12$ it holds $\forall n\in\N\colon \mathbb{I}(A_n)=1$ and the recurrence relation simplifies accordingly.
Moreover, let $n_0=\left\lfloor \frac{\ln(1/2-p)}{\ln(1-p)}\right\rfloor$. Then, $\forall n\le n_0\colon \mathbb{I}(A_n)=1$, regardless of the value of $p\in(0,1)$.
\end{remark}
\begin{remark}[Two-color card guessing and asymmetric riffle shuffle]
The random variable $C_{m_1,m_2}$ links the two-color card guessing game (with complete feedback) starting with $m_{1}$ cards of type (color) $a$ and $m_{2}$ cards of type (color) $b$ with the asymmetric riffle shuffle. Apparently, it corresponds to the guessing game for the letters of a word over the alphabet $\{a,b\}$ consisting of $m_{1}$ $a$'s and $m_{2}$ $b$'s, when $m_1,m_2$ is known to the guesses. This link was already observed in the special case $p=\frac12$, see~\cite{KuPa2025}. We refer the reader to~\cite{DiaconisGraham1981,KnoPro2001} for more background on the guessing game and to~\cite{KP2024,KuPanPro2009} for the limit laws of $C_{m_1,m_2}$.
In most works considering $C_{m_1,m_2}$ it is assumed without loss of generality that $m_1\ge m_2\ge 0$. 
However, by definition of the two-color card guessing game the order of the parameters is irrelevant under the optimal strategy: $C_{m_1,m_2}=C_{m_2,m_1}$.
\end{remark}
\begin{proof}
The proof of this recurrence is based on Lemma~\ref{lem:firstCard} and Proposition~\ref{Prop:Optimal}. We use the probabilistic description of once-shuffled decks of $n$ cards in terms a random length-$n$ words $W_1\dots W_n$, as stated in Lemma~\ref{lem:description}. We count the number of correct guesses, where we distinguish according to the first letter $W_{1}$, the value of $0<p<1$, 
as well as the numbers $n$ and $n_0$, as described in detail by the optimal strategy. 
Assume that $\frac12 \le p <1$. If $W_{1}=a$, then the first drawn card is $1$.
This card will be predicted correctly by the guesser. The person keeps his optimal strategy of guessing for the deck of remaining cards, which is order-isomorphic to a deck of $n-1$ cards generated by the random length-$(n-1)$ word $W_{2} \dots W_{n}$ of length $n-1$, independent of $W_1$. This corresponds to the summand $I_1 \big(X^{\ast}_{n-1}+\mathbb{I}(A)\big)$.
If $W_{1}=b$, then we first consider the particular case $W_1\dots W_n=b^{n}$, i.e., that the cut of the deck has been at $0$. 
Since in this case the deck of cards corresponds to the identity permutation $\text{id}_{n}$, the guesser will predict all cards correctly using the optimal strategy. This amounts to the summand $(1-I_1)I_2\cdot n$. Apart from this particular case, $W_{1}=b$ corresponds to a deck of cards where the first card is $m \ge 2$. Of course, this event causes a wrong prediction by the guesser.
However, due to complete feedback, now the guesser knows that the cut is at $m-1$, or in alternative terms, he knows that the remaining deck is generated from a word $W_{2} \dots W_{n}$ that has $j:=n-m$ $b$'s and $n-1-j=m-1$ $a$'s, with $0 \le j \le n-2$. From this point on the guesser changes the strategy, which again could be formulated in alternative terms by saying that the guesser makes a guess for the next letter in the word, in a way that the guess is $a$ if the number of $a$'s exceeds the number of $b$'s in the remaining subword, that the guess is $b$ in the opposite case, and (in order to keep the outcome deterministic) that the guess is $a$ if there is a draw between the number of $a$'s and $b$'s. These considerations yield the third summand $ (1-I_1)(1-I_2)\cdot C_{n-1-J_n,J_n}$.
For $0<p<\frac12$. we proceed by looking at the event $A_n$.
 For $\mathbb{I}(A_n)=1$ we guess the value one and proceed identically to $\frac12 \le p <1$. On the other hand,
if  $\mathbb{I}(A_n)=0$ we guess the value $\kappa_n$. If $W_{1}=a$, then the first drawn card is $1$, but
the prediction will be wrong. The person keeps his optimal strategy of guessing for the deck of remaining cards.
If $W_{1}=b$, the guesser knows that the cut is at $m-1$ and the composition of the subword $W_{2} \dots W_{n}$.
Depending on the position of the cut, the event $W_{1}=b$ corresponds to a different number. The number of correct guesses may increase, which leads to the term $(1-\mathbb{I}(A_n))\mathbb{I}(J_n=n-1-\kappa_n)$.
\end{proof}

\section{Limit law \label{sec:LimitLaw}}
\begin{theorem}         
\label{the:1}  
Given $p\in(0,1)\setminus\{\frac12\}$, let $\pe=\max\{p,1-p\}$. The shifted random variable $X_n-n\cdot p^{\ast}$, with $X_n$ counting the number of correct guessing with full feedback after an asymmetric riffle shuffle, converges in distribution to a geometric limit law:
\[
X_n -n\cdot p^{\ast} \to G, \quad \P\{G=k\}=\rho^k(1-\rho),\quad k\ge 0,
\]
with $\rho=(1-\pe)/\pe$.
\end{theorem}

\begin{remark}

For $p\in\{0,\frac12,1\}$ the value of $\rho$ degenerates to one and thus the limit law has to be different; see the introduction for a summary of the result for $p=\frac12$~\cite{KuPa2025},
as well as the boundary case.
\end{remark}
A crucial ingredient for the proof of Theorem~\ref{the:1} is the limit law for $C_{m_1,m_2}$, when $m_2$ is proportional to $m_1$. 
This is then used to understand the random variable $C_{n-1-J_n,J_n}$\eqref{eqn:Xn_DistEqn}. 
\begin{lem}
\label{lem:1}
Given $p\in (0,1)\setminus\{\frac12\}$, let $p^{\ast}=\max\{p,1-p\}$. The random variable $C_{n-1-J_n,J_n}$ has, for $n\to \infty$ a geometric limit law:
\[
C_{n-1-J_n,J_n}- n\cdot p^{\ast}\to G,\quad \P\{G=k\}=\rho^k(1-\rho),\quad k\ge 0.
\]
with $\rho=(1-p^{\ast})/p^{\ast}$.
\end{lem}
\begin{remark}[Limit law and a fixed-point equation]
It turns out that $C_{n-1-J_n,J_n}$ has the same limit law as $X_n$. 
In order to shed more light on the fact that the limit laws of $C_{n-1-J_n,J_n}$ and $X_n$ coincide, 
we give the following non-rigorous motivation (compare with~\cite{KuPa2025} for the case $p=1/2$):
starting from the distributional equation~\eqref{eqn:Xn_DistEqn}, we subtract on both sides $n p^{\ast}$. 
The shifted random variable $Y_n=X_n-n p^{\ast}$ satisfies for large $n$ and any $0<p<1$, omitting asymptotically negligible terms, the equation
\[
Y_n \sim I_1 \cdot Y^{\ast}_{n-1} + (1-I_1)\big(C_{n-1-J_n,J_n}-np^{\ast}\big),
\]
where $I_1=\Be(p)$. Thus, using the fact that $C_{n-1-J_n,J_n}-np^{\ast}\to G$ and the assumption that $Y_n$ convergences to a limit law $Y$ would imply a sort of fixed-point equation for $Y$:
\[
Y \law I_1 \cdot Y^{\ast} + (1-I_1)\cdot G, 
\]
where $Y^{\ast}\law Y$ denotes an independent copy of $Y$ and $G$, which has (at least) the solution $Y\law G$.
Of course, making this argument rigorous would require more effort, compare with the contraction method~\cite{ContrNR,ContrRR}. 
\end{remark}

As we will later require all the different ranges of $m_1,m_2$, we collect here the results of~\cite{KuPa2025,KuPanPro2009}
\begin{theorem}[\cite{KuPa2025,KuPanPro2009}]
\label{the:KPP}
The random variable $C_{m_1,m_2}$ counting the number of correct guesses in the card
guessing game starting with $m_1$ red and $m_2$ black cards has for $m_1\ge m_2$, with $m_1\to \infty$,
the following limit laws, with $\hat{C}_{m_1,m_2}=C_{m_1,m_2}-m_1$ and $\Delta_m=m_1-m_2$

\begin{itemize}
\item The region $m_2=o(m_1)$: the centered random variable $\hat{C}_{m_1,m_2}$ is asymptotically zero:
		\begin{equation*}
		\P\{\hat{C}_{m_1,m_2}=0\}=1-\frac{m_2}{m_1+1}\sim 1.
		\end{equation*}
\item The region $m_2\sim\rho m_1$, with $0<\rho<1$: the centered random variable $C_{m_1,m_2}^{\ast}$ is asymptotically
geometrically distributed with parameter $\rho$:
		\begin{equation*}
		\P\{\hat{C}_{m_1,m_2}=k\}\sim(1-\rho)\rho^k, \quad k\in\N_0.
		\end{equation*}
\item The region $m_1=o(\Delta_m^2)$: the centered and scaled random variable
$\frac{\Delta_m}{m_1}\hat{C}_{m_1,m_2}$ is asymptotically exponential distributed:
\begin{equation*}
\frac{m_1}{\Delta_m}\P\{\hat{C}_{m_1,m_2}^{\ast}=x\} \sim e^{-x},\quad x\ge 0.
\end{equation*}

\item The region $\Delta_m \sim \rho\sqrt{m_1}$ and $\rho > 0$: the centered and scaled random variable
$\hat{C}_{m_1,m_2}^{\ast}\sqrt{m_1}$ weakly converges to a $\text{LinExp}$ distributed random variable $L_{\rho}$ with density $f_{L_{\rho}}(x)$,
\begin{equation*}
\sqrt{m_1}\,\P\{\hat{C}_{m_1,m_2}^{\ast}=x\} \sim f_{L_{\rho}}(x)= (\rho+2x)e^{-x(\rho+x)},\quad x\ge 0.
\end{equation*}
\item The region $\Delta_m=o(\sqrt{m_1})$: the centered and scaled random variable
$\hat{C}_{m,n}^{\ast}\sqrt{m}$ is asymptotically Rayleigh distributed:
\begin{equation*}
\sqrt{m_1}\,\P\{\hat{C}_{m_1,m_2}=x\} \sim 2xe^{-x^2},\quad x\ge 0.
\end{equation*}
\end{itemize}
\end{theorem}

\begin{proof}[Proof of Lemma~\ref{lem:1}]
We need to analyze the asymptotics of $C_{n-1-J_n,J_n}$, as $n$ tends to infinity. 
For this, we use the classical normal approximation of the binomial distribution,
\begin{equation}
\label{eq:AsympBin}
\P\{J_n = k\}\sim  \frac{1}{\sigma_n\sqrt{2\pi}}\exp\Big(\frac{(k-\mu_n)^2}{2\sigma_n^2}\Big),
\end{equation}
with 
\[
\mu_n=n\cdot p,\quad\sigma_n^2=n p (1-p),
\]
such that for real $t$ and $n$ tending to infinity
\begin{equation}
\label{eqn:ExpansionBinomial}
\P\{J_n = n\cdot p+t\sqrt{n}p(1-p)\}\sim \frac{1}{\sqrt{2\pi n p(1-p)}}e^{-t^2/2}.
\end{equation}
Correspondingly, we turn to $C_{n-1-J_n,J_n}$, conditioned on $J_n=n\cdot p+t\sqrt{n}p(1-p)$.
We have
\[
\P\{C_{n-1-J_n,J_n}= n\cdot p^{\ast}+k\}
=\sum_{j=0}^{n-1}\P\{J_n=j\}\P\{C_{n-1-j,j}=n\cdot p^{\ast}+k\}.
\]
Approximation of the sum by the Euler-MacLaurin formula and by~\eqref{eqn:ExpansionBinomial} we get
\[
\P\{C_{n-1-J_n,J_n}= n\cdot p+k\}=
\int_1^{n-1} \P\{J_n = j\}\P\{C_{n-1-j,j}=n\cdot p+k\}dj.
\]
We split the integral over the three intervals $[1,(1-\delta)n\cdot p)$, $[(1-\delta)n\cdot p,(1+\delta)n\cdot p]$
and $((1+\delta)n\cdot p^{\ast},n]$ for an arbitrarily small but fixed $0<\delta<1$. For $[(1-\delta)n\cdot p,(1+\delta)n\cdot p]$ we get the dominant part with $j=j(t)= n\cdot p+t\sqrt{n p(1-p)}$ and 
$t\in\R$, asymptotically. The remaining parts are negligible by Chernoff bounds for the binomial distribution, $B=B(n,p)$:
\[
\P\{B\ge (1+\delta)np\}\le e^{-\frac{n p\delta^2}{2+\delta}},
\quad \P\{B\le (1-\delta)np\}\le e^{-\frac{n p\delta^2}{2}}.
\]
By dominated convergence, we obtain
\[
\P\{C_{n-1-J_n,J_n}= n\cdot p^{\ast}+k\}
\sim \int_{-c_n}^{c_n} \frac{1}{\sqrt{2\pi}}e^{-t^2/2}\P\{C_{n-1-j,j}=n\cdot p+k\}dt,
\]
with $j=j(t)$ and $c_n=\delta \sqrt{n}/\sqrt{p/(1-p}$. Noting that
\[
C_{n-1-j,j}=C_{n(1-p)-t\sqrt{n}p(1-p)-1,n\cdot p+t\sqrt{n}p(1-p)},
\]
we observe that by Theorem~\ref{the:KPP} it holds for $t\ll\sqrt{n}$ the estimate
\[
\P\{C_{n-1-j,j}=n\cdot p+k\}\to \rho^k(1-\rho),\quad \rho=\frac{1-p^{\ast}}{p^{\ast}},
\]
independent of $t$. We choose some small $0<\epsilon<\frac12$ and split again the integration interval into
\[
[c_n,c_n]=[-c_n,-n^{\frac12-\epsilon})\cup [-n^{\frac12-\epsilon},n^{\frac12-\epsilon}]\cup(n^{\frac12-\epsilon},c_n].
\]
Two integrals are asymptotically negligible, as
\[
\int_{-c_n}^{-n^{\frac12-\epsilon}} \frac{1}{\sqrt{2\pi}}e^{-t^2/2}\P\{C_{n-1-j,j}=n\cdot p+k\}dt
\le \int_{-\infty}^{-n^{\frac12-\epsilon}} \frac{1}{\sqrt{2\pi}}e^{-t^2/2}dt,
\]
which tends to zero for $n\to\infty$. In the center part we have $t\ll\sqrt{n}$ and use our estimate for $\P\{C_{n-1-j,j}=n\cdot p+k\}$.
Completing the tails for the center and integration over $(-\infty,\infty)$ leads to the stated result.
\end{proof}
Now we are ready to prove our main result.
\begin{proof}[Proof of Theorem~\ref{the:1}]
The starting point for the derivation of the limit law is the distributional equation~\eqref{eqn:Xn_DistEqn}. 
To prove this theorem, we turn first to the case $\frac12<p<1$, such that $\forall n\in\N\colon \mathbb{I}(A_n)=1$. The distributional equation~\ref{eqn:Xn_DistEqn} simplifies to
\begin{equation}
\label{eqn:XnDistSimple}
\begin{split}
X_n &\law I_1 \big(X^{\ast}_{n-1}+1\big)+(1-I_1)\big(I_2\cdot n +(1-I_2)C_{n-1-J_n,J_n}\big).
\end{split}
\end{equation}
Note in passing that the contribution of $(1-I_1)I_2\cdot n$ is asymptotically negligible.
Iteration of the distributional equation leads to a geometric waiting time until the first draw of a letter $b$.
According to Theorem~\ref{the:distDecomp} we get for $n\to\infty$, regardless of the value of $0<p<1$ the equation
\begin{align*}
 \P\{X_n=n-k\} &= p\cdot \P\{X_{n-1}+1=k\}\\
 &\quad+ (1-p)\big((1-(1-p)^{n-1}\big)\P\{C_{n-1-J_n,J_n}=n-k\}.
\end{align*}
Moreover, by iterating this recursive representation we observe that for $n\to\infty$,
\[
\P\{X_n=n-k\} 
\sim \sum_{\ell \ge 1}p^{\ell} \cdot
\P\big\{C_{n-\ell-J_{n+1-\ell},J_{n+1-\ell}}=n-k\big\}.
\]
The argument used before extends to that all $C_{n-\ell-J_{n+1-\ell},J_{n+1-\ell}}$, $\ell\ge 1$ fixed, and the probability mass function
converges. It turns out that the same argument works for $0<p<\frac12$, as for $n\to\infty$ the optimal strategy is guessing the number one, and 
$\mathbb{I}(A_n)=1$.
\end{proof}

\section{More limit laws and phase transitions\label{sec:Boundary}}
We already encountered the three different limit laws for $p\in [0,1]$. Obviously, for $p\in\{0,1\}$, the distribution of $X_n$ degenerates and $\P\{X_n=n\}=1$, whereas for $p=\frac12$ we have a completely different behavior, as the fluctuations around $n/2$ are of order $\sqrt{n}$. This sparks interest in phase transitions for $X_n$, depending on how quickly $p=p(n)$ approaches the boundary for large $n$, or the how quickly $p$ approaches $\frac12$. It will turn out that at the boundary, the transitions are mainly governed by the binomial distribution, whereas for $p$ close to one-half, the transitions are governed by the two-color stage. We only focus on the case $p\to 1$, as the case $p\to 0$ is (nearly) symmetric.
\subsection{Phase transition at the value one-half\label{ssec:oneH}}
We assume that $p=\frac12 + \alpha_n$, with $\alpha_n\to 0$, and set $\alpha_n \sim \frac{b}{n^c}$, with constants $b\in\R$ and $c>0$.
In the following, we distinguish between different intervals for $c$, which leads to transitions for the limit laws of $X_n$,
as the expectation of the binomial distribution $B_n\law B(n,p)$ and its truncated version $J_n$ changes:
\[
\mu_n=\E(B_n)=n\cdot p \sim \frac{n}2 + b\cdot n^{1-c},\quad 
\sigma_n=\sqrt{n\cdot p \cdot (1-p)} \sim \frac12\sqrt{n}.
\]
This influences the limit law of the random variable $C_{n-1-J_n,J_n}$.
\begin{lem}
\label{Prop:twoColorBinLimit}
The random variable $C_{n-1-J_n,J_n}$ has the following limit laws for $p=p(n)=\frac12 + \alpha_n$, with $\alpha_n \sim \frac{b}{n^c}\to 0$,
as $n$ tends to infinity, with constants $b\in\R$, $c>0$:
\[
\frac{C_{n-1-J_n,J_n}-\frac{n}2 }{\beta_n}\claw
\begin{cases}
\Exp(\frac{1}{2|b|})+|b|,\quad \text{for }0<c<\frac12,\\[0.25cm]
\frac12\chi(3,2|b|),\quad \text{for }c= \frac12,\\[0.25cm]
GG,\quad \text{for }c\ge \frac12,
\end{cases}
\]
with $\beta_n=n^{1-c}$ for $0<c<\frac12$ and $\beta_n=\sqrt{n}$ for $c\ge \frac12$. Here, $\Exp(\lambda)$ denotes an Exponential distribution with parameter $\lambda$, $GG$ denotes a generalized gamma distributed random variable, with density $f_{GG}(x)=\sqrt{\frac{2}{\pi}}\cdot 8x^2e^{-2x^2}$, $x\ge 0$, and $\chi(k,\lambda)$ denotes a non-central Chi-distribution with parameter $\lambda$ and $k$ degrees of freedom.
\end{lem}
\begin{remark}[Special case $b=0$]
The density of $\frac12\chi(3,\sqrt{2}\cdot |b|)$ is given by
\[
f_{\chi}(x)=\frac{4x}{\sqrt{2\pi}}e^{-2(x^2+b^2)}\cdot\frac{\sinh(4bx)}b,\quad x\ge 0,
\]
which is an even function in $b$. For $b\to 0$, it converges by l'H\^opital's rule to the density $f_{GG}(x)$ 
of the generalized gamma distribution. Thus, for $c=\frac12$ we can interchange the limits $n$ to infinity with the limit $b$ to zero.
\end{remark}

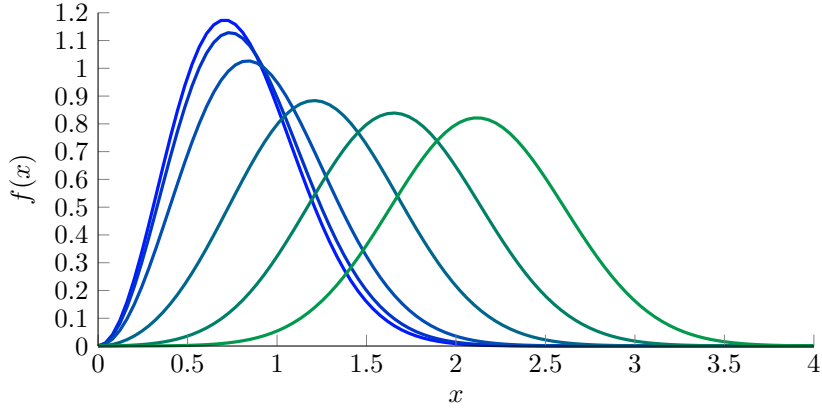
\begin{figure}[!htb]
\begin{tikzpicture}
\begin{axis}[
 domain=0:4,samples=100,
  axis lines*=left,
  height=6cm, width=11cm,
  xtick={0,0.5, 1,1.5, 2,2.5, 3,3.5, 4,4.5,5,5.5}, 
	 ytick={0,0.1, 0.2, 0.3, 0.4, 0.5, 0.6, 0.7,0.8,0.9,1,1.1,1.2}, 
	ymin=0, ymax=1.2,
	xmin=0, xmax=4,
   axis on top,  clip=false, no markers,
	 xlabel={$x$},
   ylabel={$f(x)$},
  ]
 \addplot [very thick, myblauA] {sqrt(2/pi)*8*x^2*exp(-2*x^2)}; 
\addplot [very thick, myblauB] {4*x*exp(-2*x^2-2*0.25^2)*(exp(4*0.25*x)-exp(-4*0.25*x))/(0.25*2*sqrt(2*pi))}; 
 \addplot [very thick, myblauC] {2*4*x*exp(-2*x^2-0.5)*(exp(2*x)-exp(-2*x))/(2*sqrt(2*pi))}; 
 \addplot [very thick, myblauD] {4*x*exp(-2*x^2-2)*(exp(4*x)-exp(-4*x))/(2*sqrt(2*pi))}; 
 \addplot [very thick, myblauE] {4*x*exp(-2*x^2-2*1.5^2)*(exp(4*1.5*x)-exp(-4*1.5*x))/(1.5*2*sqrt(2*pi))}; 
\addplot [very thick, myblauF] {4*x*exp(-2*x^2-8)*(exp(8*x)-exp(-8*x))/(2*2*sqrt(2*pi))}; 
\end{axis}
\end{tikzpicture}
\caption{Plot of the density functions $f_{\chi}(x)$ of $\frac12\chi(3,\sqrt{2}\cdot |b|)$ and the generalized Gamma distribution $f_{GG}(x)$ (case $b=0$ of $f_{\chi}(x)$), for $b\in\{0,0.25,0.5,1,1.5,2\}$ occurring in Proposition~\ref{Prop:twoColorBinLimit}
and Theorem~\ref{the:Half}.}
\label{fig:Densities}%
\end{figure}

The limit law of $X_n$ changes accordingly to the limit law of $C_{n-1-J_n,J_n}$
\begin{theorem}
\label{the:Half}
The random variable $X_n$, counting the number of correct guessing with full feedback after an asymmetric riffle shuffle, has the following limit laws for $p=p(n)=\frac12 + \alpha_n$, with $\alpha_n \sim \frac{b}{n^c}\to 0$, as $n$ tends to infinity, with constants $b\in\R$, $c>0$:
\[
\frac{X_n-\frac{n}2}{\beta_n}\claw
\begin{cases}
\Exp(\frac{1}{2|b|})+|b|,\quad \text{for }0<c<\frac12,\\[0.25cm]
\frac12\chi(3,2\cdot |b|),\quad \text{for }c= \frac12,\\[0.25cm]
GG,\quad \text{for }c\ge\frac12,\\
\end{cases}
\]
with $\beta_n=\max\{n^{1-c},\sqrt{n}\}$.
\end{theorem}

\begin{proof}[Proof of Lemma~\ref{Prop:twoColorBinLimit}]
We consider the distribution function 
\[
F_n(x)=\P\big\{C_{n-1-J_n,J_n}\le n\cdot p + x\beta_n\big\}
\]
for fixed positive real $x$. Conditioning on the truncated binomial distribution gives
\[
F_n(x) =\sum_{j=1}^{n-1}\P\big\{C_{n-1-j,j}\le n\cdot p + x\beta_n\big\}\P\{J_n=j\}.
\]
Our task is to evaluate this sum for $n$ tending to infinity. We proceed using standard methods.
First, we proceed as in the proof of Lemma~\ref{lem:1} and split the integral into three parts, focusing on the dominant contribution close to the mean, as the rest is negligible.
There, we use the local limit theorem for the binomial distribution $\P\{J_n=j\}$~\eqref{eq:AsympBin}, as well as approximating the sum by an integral. Then, we also require the asymptotics of the two-color card-guessing game in the required ranges, as stated in Theorem~\ref{the:KPP}. Combining all these three steps will yield the desired asymptotic expansion of $F_n(x)$. 

By approximating the sum by an integral, Chernoff bounds, and the local limit theorem for the binomial distribution, we get for large $n$ the asymptotics
\[
F_n(x)\sim \int_{-\delta \mu_n/\sigma_n}^{\delta \mu_n/\sigma_n}\P\big\{C_{n(1-p) -t\sigma_n -1,n p+t\sigma_n}\le \frac{n}2 + x\beta_n\big\} \cdot \frac{e^{-\frac{t^2}{2}}}{\sqrt{2\pi}} dt.
\]
Next, insert the asymptotics of $\mu_n$: 
\[
C_{m_1,m_2}=C_{n(1-p) -t\sigma_n -1,n p+t\sigma_n}
\]
with $m_1=m_1(n)$, $m_2=m_2(n)$, such that
\[
m_1=\frac{n}2 + b\cdot n^{1-c} + t\sqrt{n}/2,\quad m_2=n/2 -b\cdot n^{1-c}- t\sqrt{n}/2 -1.
\]
This implies that
\[
\Delta_m=2bn^{1-c}+t\sqrt{n}+1
\]
is at least of order $\sqrt{n}$. We obtain the relations
\[
		\begin{array}{ll}
       m_1=o(|\Delta_m|^2), & \text{for }0<c<\frac12,\\[0.2cm]
       |\Delta_m|\sim |2b+t|\sqrt{m_1}, & \text{for } c=\frac12,\\[0.2cm]
        \Delta_m=t\sqrt{m_1}& \text{for } c>\frac12.
        \end{array}
\]
We begin with the case $0<c<\frac12$. Here, we use an approximation of the integrand using the exponential limit law for the two-color card guessing stage. Before we do so, we split for some fixed $0<\epsilon<\frac12-c$ the interval again into three parts:
\[
[-\delta \mu_n/\sigma_n,\delta \mu_n/\sigma_n]=
[-\delta \mu_n/\sigma_n,-n^{\frac12-\epsilon})\cup [-n^{\frac12-\epsilon},n^{\frac12-\epsilon}] \cup (n^{\frac12-\epsilon},-\delta \mu_n/\sigma_n].
\] 
Two integrals are asymptotically negligible, but the center part we have $t\sqrt{n}\ll n^{1-c}$ and we can apply the asymptotics
for $C_{m_1,m_2}$. We obtain
\begin{equation*}
\begin{split}
&P\{C_{n(1-p) -t\sigma_n -1,n p+t\sigma_n}\le \frac{n}2+n^{1-c}x\}\\
&\quad=\P\{ \frac{C_{n(1-p) -t\sigma_n -1,n p+t\sigma_n} - \frac{n}2 -|b|n^{1-c} }{2|b|n^{1-c}}\le \frac{x}{2|b|}-\frac12\}\\
&\quad\to 1-e^{-x/(2|b|)-\frac12},
\end{split}
\end{equation*}
independent of $t$. Consequently,
\[
F_n(x)\sim \int_{-n^{\frac12-\epsilon}}^{n^{\frac12-\epsilon}}\big(1-e^{-x/(2|b|)-\frac12}\big)\cdot \frac{e^{-\frac{t^2}{2}}}{\sqrt{2\pi}} dt.
\]
Completing the tails for the center and integration over $(-\infty,\infty)$ leads to
\[
F_n(x)\to 1-\exp\Big(-\frac{x}{(2|b|}-\frac12\Big).
\]
Thus, we finally get an exponential limit law with parameter $1/2|b|$, shifted by $|b|$. For $c= \frac12$, we note that $C_{m_{1},m_{2}} \ge \max\{m_{1},m_{2}\}$. We deduce that for $t>2x-2b$ or $t<-2x-2b$ it holds
\[
\P\big\{C_{n/2 -b\cdot n^{1-c}- t\sqrt{n}/2 -1,\frac{n}2 + b\cdot n^{1-c} + t\sqrt{n}/2}\le \frac{n}2 + x\beta_n\big\}= 0.
\]
This leads to
\[
F_n(x)\sim \int_{-2x-2b}^{2x-2b}\P\big\{C_{n(1-p) -t\sigma_n -1,n p+t\sigma_n}\le \frac{n}2 + x\beta_n\big\} \cdot \frac{e^{-\frac{t^2}{2}}}{\sqrt{2\pi}} dt.
\]
We use the substitution $t=u-2b$, $u=t+2b$ to get
\[
F_n(x)\sim \int_{-2x}^{2x}\P\big\{C_{\frac{n}2 -u\sigma_n -1,\frac{n}2+u\sigma_n}\le \frac{n}2 + x\beta_n\big\} \cdot \frac{e^{-\frac{(u-2b)^2}{2}}}{\sqrt{2\pi}} du.
\]
We use Theorem~\ref{the:KPP}, by setting $\rho = \sqrt{2}u$ to get 
\begin{multline*}
\P\big\{C_{\frac{n}2 -u\sigma_n -1,\frac{n}2+u\sigma_n}\le \frac{n}2 + x\beta_n\big\}
\sim 1-\exp\Big(-\sqrt{2}\big(x-\frac{t}2\big)\big(\sqrt{2} u + \sqrt{2}(x-\frac{u}{2})\big)\Big),
\end{multline*}
which simplifies to
\[
 1-\exp\Big(-2x^2+\frac{u^2}2\Big).
\]
This implies that
\[
F_n(x)\to F(x)=\frac{1}{\sqrt{2\pi}}\cdot \int_{-2x}^{2x}e^{-(u-2b)^2/2}\Big(1-e^{-2x^2+\frac{u^2}2}\Big)du
\]
Note that this expression is an even function in $b$. 
We rewrite the limit as
\begin{align*}
F(x)&=\frac{1}{\sqrt{2\pi}}\cdot \int_{-2x}^{2x}e^{-(u-2b)^2/2}dx - \frac{e^{-2x^2-2b^2}}{\sqrt{2\pi}}\cdot \int_{-2x}^{2x}e^{-2bu}du\\
&=\frac{1}{\sqrt{2\pi}}\cdot \int_{-2x}^{2x}e^{-(u-2b)^2/2}dx  - \frac{e^{-2x^2-2b^2}}{\sqrt{2\pi}}\cdot \frac{\sinh(4bx)}{b}.
\end{align*}
Next, we differentiate the last expression with respect to $x$. The first integral leads to
\[
\frac{2}{\sqrt{2\pi}}\cdot \big(e^{-(2x-2b)^2/2}+e^{-(-2x-2b)^2/2}\big)
=\frac{4e^{-2x^2-2b^2}}{\sqrt{2\pi}}\cdot \cosh(4bx).
\]
This term cancels with the second summand of the product rule, leaving only the stated density function.
Finally, we note that the non-central chi distribution with parameter $\lambda$ and $k$ degrees of freedom has the density function
\[
\frac{e^{-x^2/2-\lambda^2/2}x^k \lambda}{(\lambda x)^{k/2}} I_{k/2-1}(\lambda x),
\]
where $I_{\alpha(z)}$ is a modified Bessel function of the first kind:
\[
I_{\alpha}(x) = \sum_{m=0}^\infty \frac{1}{m!\, \Gamma(m+\alpha+1)}\left(\frac{x}{2}\right)^{2m+\alpha}
\]
For $k=3$ and thus $\alpha=\frac12$, it simplifies to a hyperbolic sine,
\[
I_{1/2}(x)=\sum_{m\ge 0}\frac{2\cdot 4^m}{\sqrt{\pi}(2m+1)!}\cdot \left(\frac{x}{2}\right)^{2m+\frac12}
=\frac{2}{\sqrt{2\pi x}}\sum_{m\ge 0}\frac{x^{2m+1}}{(2m+1)!}=\frac{2}{\sqrt{2\pi x}}\sinh(x),
\]
noting that 
\[
\frac{1}{m!\Gamma(m+\frac32)}=\frac{2^{m+1}}{m!(2m+1)!!\Gamma(\frac12)}=\frac{2\cdot 4^m}{\sqrt{\pi}(2m+1)!}.
\]
Next, we set $\lambda=2b$. As for $C\law \chi(3,2\cdot |b|)$ we have to study $\P\{\frac{C}2\le x\}$, our density 
is given by $2\cdot f_C(2x)$, which leads to the stated expression after some simplifications. Finally, the case $c>\frac12$ is similar to the boundary case $b=0$, leading to a Maxwell-Boltzmann limit law, a special instance of the generalized gamma distribution~\cite{KuPa2025}.
\end{proof}

Now we are ready to prove the corresponding result for $X_n$.
\begin{proof}[Proof of Theorem~\ref{the:Half}]
We proceed very similar to the proof of Theorem~\ref{the:1}, using the distributional equation~\eqref{eqn:Xn_DistEqn}. 
By Theorem~\ref{the:distDecomp} we get for $n\to\infty$ and $p=\frac12 + \alpha_n$, with $\alpha_n\to 0$, regardless of $p<\frac12$ or $p>\frac12$, the equation
\begin{align*}
 \P\{X_n\le n\cdot p + x\beta_n\} &= p\cdot \P\{X_{n-1}\le n\cdot p + x\beta_n-1\}\\
 &\quad+ (1-p)\big((1-(1-p)^{n-1}\big)\P\{C_{n-1-J_n,J_n}\le n\cdot p + x\beta_n\}.
\end{align*}
Moreover, by iterating this recursive representation we observe that for $n\to\infty$,
\[
\P\{X_n=\le n\cdot p + x\beta_n\} 
\sim \sum_{\ell \ge 1}p^{\ell} \cdot
\P\big\{C_{n-\ell-J_{n+1-\ell},J_{n+1-\ell}}=\le n\cdot p + x\beta_n-(\ell-1)\big\}.
\]
The argument used before extends to that all $C_{n-\ell-J_{n+1-\ell},J_{n+1-\ell}}$, $\ell\ge 1$ fixed. 
The shift from $(n-\ell)\cdot p+x\beta_{n-\ell}$ to $n\cdot p + x\beta_n-(\ell-1)$ does not change the limit law, 
which proves the stated result, as $\sum_{\ell \ge 1}p^{\ell}\to 1$ for $p\to\frac12$.
\end{proof}

\subsection{Phase transition at the boundary\label{ssec:Boundary}}
We focus our interest on $p\to 1$. The case $p\to 0$ is nearly symmetric and leads to a similar result; hence, we refrain from stating this case explicitly. In order to obtain an intuition for the behavior at the boundary, we look at the measure $\Rif_p$ itself.
\begin{prop}
\label{prop:bound1}
The probability of $\id_n\in\mathcal{S}_n$ under the asymmetric Gilbert-Shannon-Reed model $\Rif_p$ has the following behavior 
for $p$ near one, with $p=1-\alpha_n$, $\alpha_n>0$ and $\alpha_n\to 0$:
  \[
    \Rif_p(\id_n)\longrightarrow 
		\left\{
		\begin{array}{ll}
        1, & \text{for }\alpha_n\ll \frac1n,\\
        e^{-\lambda}, & \text{for }\alpha_n\sim \frac{\lambda}{n},\ \lambda>0,\\
        0, & \text{for } \alpha_n\gg \frac1n.
        \end{array}\right.
				\]
\end{prop}
\begin{proof}
From Definition~\ref{eqn:defRif} we get
\[
\Rif_p(\id_n)=\frac{(1-p)^{n+1}-p^{n+1}}{1-2p}
= \frac{(1-\alpha_n)^{n+1}-\alpha_n^{n+1}}{1-2\alpha_n}.
\]
Then, the standard $\exp-\ln$ representation and expansion of $\ln(1-x)$ at $x=0$ leads to the stated result:
\begin{align*}
(1-\alpha_n)^{n+1}=\exp\Big((n+1)\cdot \ln(1-\alpha_n)\Big)
=\exp(-n\cdot \alpha_n)\cdot \Big(1+\gro\big(\frac{\alpha_n^2}{n}\big) +\gro\big( \alpha_n\big)\Big).
\end{align*}
\end{proof}
By Proposition~\ref{prop:bound1}, we anticipate a transition for the limit law of $X_n$, 
as 
\[
\P\{X_n=n\}=\Rif_p(\id_n).
\] For the sake of both simplicity and convenience, we set $\alpha_n=\frac{\lambda}{n^{c}}$, $c>0$ and distinguish according to the value $c$. 
We will observe a classical phase transition from a degenerated limit law, to Poisson, to a central limit theorem, often encountered in the analysis of random structures. Here, this phase transition is essentially directly driven by an underlying binomial distribution.
\begin{theorem}[Phase transition near $p=1$]
\label{the:PoN}
For $p$ near one, such that $p=1-\alpha_n$, with $\alpha_n=\frac{\lambda}{n^c}$ and constants $\lambda,c>0$, we observe
the following results for the number of correct guesses $X_n$ in a once-asymmetrically riffle-shuffled deck of $n$ cards with full feedback:
\[
		\begin{array}{ll}
        n-X_n\to 0, & \text{for }c>1,\\[0.2cm]
        n-X_n\to\Po(\lambda), & \text{for }c=1,\\[0.2cm]
				\frac{n-X_n - \lambda n^{1-c}}{\sqrt{\lambda n^{1-c}}}\to \mathcal{N}(0,1), & \text{for } c<1.\\[0.2cm]
        \end{array}
				\]
\end{theorem}
\begin{proof}
For $c>1$ the result follows directly from Proposition~\ref{prop:bound1}. Let $E_{n,k}$ denote the event the we observe $k$ times the letter $b$ and $n-k$ times the letter $a$ in the random word $W_1\dots W_n$, with $k\in\N$ independent of $n$, excluding the case $a^{n-k}b^k$. Then, by definition of the riffle shuffle, we have
\[
\P\{E_{n,k}\}=\Big(\binom{n}{k}-1\Big)(1-p)^k p^{n-k}.
\]
Thus, for $p=1-\alpha_n$, $\alpha_n\sim \frac{\lambda}{n}$ we have 
\[
\P\{E_{n,k}\}\to \frac{\lambda^k}{k!}e^{-\lambda}, \quad k\ge 0,
\]
by the classical Poisson limit law for the binomial distribution, or directly by application of Stirling's formula for the factorial. 
We have $\P\{X_n=n-k\}\sim \P\{E_{n,k}\}$. This is due to the fact that by the optimal strategy we guess for large $n$ at least $n-k$ correct, but also not more. As we have exponential waiting times for the occurrences of $b$, in particular the waiting time for $k$ letters $b$ tends to the sum of $k$ independent exponentials, also known as an Erlang distribution $\Erl(\lambda, k)$.
This implies the probability that $E_{n,k}$ leads to more than $n-k$ successes is asymptotically negligible, as we always guess the increasing subsequence originally stemming from the $n-k$ letters $a$. We point out an alternative more direct derivation of the Poisson limit law by using the distributional equation~\eqref{eqn:Xn_DistEqn}. Note that, even in the simplified form~\eqref{eqn:XnDistSimple}, the argument and the calculations are slightly involved. The reason for this are the several layers of randomness involved: first, an exponential waiting time for the first occurrence of $b$ as the limit law for the geometric distribution; second, the Poisson limit law of $J_n$; third, the degenerate limit law of $C_{n-1-J_n,J_n}$ by Theorem~\ref{the:KPP}. We outline this argument in the following.
Let us start with the case $k=0$. We already observed that $\Rif_p(\id)\to e^{-\lambda}$. 
For $k\ge 0$ we obtain from the distributional equation the representation
\begin{equation}
\begin{split}
\label{eq:DistPoisson}
\P\{n-X_n=k\}&=\mathbb{I}(\{k=0\})\cdot\Rif_p(\id)\\
&\quad+\sum_{\ell=0}^{n-1}p^{\ell}(1-p)\big(1-(1-p)^{n-1}\big)\P\{C_{n-1-\ell-J_{n-\ell},J_{n-\ell}}=n-k-\ell\}.
\end{split}
\end{equation}
Assume first that $k\ge 1$. The geometric waiting time for the first occurrence of $b$ approaches for $n\to\infty$, $p= 1-\frac{\lambda}n$ an exponential distribution.
As outlined, we approximate the sum by an integral. A (truncated) binomial distribution $J_N$ with $N=n(1-t)\to\infty$, $t\in(0,1)$ and parameter $p$ has a Poisson limit law: $N-J_N\to \Po((1-t)\lambda)$. Combining the steps leads for $k\ge 1$ to
\[
\P\{n-X_n=k\}\sim \int_0^{1}\Big(1-\frac{\lambda}n\Big)^{\ell}\lambda \P\{C_{n-1-\ell-J_{n-\ell},J_{n-\ell}}=n-k-\ell\}dt,
\]
with $\ell=n\cdot t$. A closer inspection of $C_{n-1-\ell-J_{n-\ell},J_{n-\ell}}$ gives
\begin{equation*}
\begin{split}
&\P\{C_{n-1-\ell-J_{n-\ell},J_{n-\ell}}=n-k-\ell\}\\
&\quad\sim\sum_{r\ge 0}\P\{J_{n-\ell}=n-1-\ell-r\}\P\{C_{r,n-\ell-r-1}=n-k-\ell\}.
\end{split}
\end{equation*}
The only contribution here comes from the case $r=k-1$, as by Theorem~\ref{the:KPP} all other summands tend to zero.
Thus, we obtain further
\[
\P\{n-X_n=k\}\sim \int_0^{1}\Big(1-\frac{\lambda}n\Big)^{nt}\lambda \frac{(1-t)^{k-1} \lambda^{k-1}}{(k-1)!}e^{-\lambda(1-t)}dt=
\frac{\lambda^k}{k!}e^{-\lambda}.
\]
For $k=0$ we do not have a contribution from the sum in~\eqref{eq:DistPoisson}, but only the $\Rif_p(\id)$, whose asymptotic we already know from Proposition~\ref{prop:bound1}. In order to obtain the central limit theorem, we proceed in a slightly different way turning to characteristic function. This avoids the more cumbersome expressions
from the expansion of the distribution functions. Set $Y_n=n-X_n$. The distributional equation~\eqref{eqn:Xn_DistEqn}
of $X_n$ is readily translated and we get
\begin{equation}
\label{{eqn:Yn_DistEqn}}
\begin{split}
Y_n &\law I_1\cdot  Y^{\ast}_{n-1}+ (1-I_1)(1-I_2)\cdot \big(n-C_{n-1-J_n,J_n}\Big), \quad n\ge 2,
\end{split}
\end{equation}
where $Y^{\ast}$ denotes an independent copy of $Y$. The initial condition is $Y_1=0$. 
Recall that $J_n\law B_{\text{pos}}(n-1,p)$, such that the total variation distance is given by
\[
d_{\text{TV}}(J_n;B(n-1,p))=
(1-p)^{n-1}
\]
For large $n$, we have $J_n\sim \mu_n + \sigma_n\cdot\mathcal{N}(0,1)$.
For $B(n,p)$ this follows classically by either characteristic functions of by checking Lindeberg's condition.
Here, $\mu_n$ and $\sigma_n$ are given by
\[
\mu_n= n\cdot p= n-\lambda n^{1-c},\quad \sigma_n\sim \sqrt{n p(1-p)} =
\sqrt{\lambda n^{1-c}(1-\lambda n^{-c}}\sim \sqrt{\lambda} n^{\frac{1-c}2}.
\]
By our assumption $\mu_n\to \infty$ and $J_n$ satisfies a central limit theorem. 
This extends to $J_N$ for any $N=N(n)\to\infty$.
By Theorem~\ref{the:KPP}, it holds $C_{n-1-J_n,J_n}\approx J_n$, which also motivates the additional centering by $n\alpha_n$, as well as the scaling by $\sqrt{n\alpha_n}$ in the stated central limit theorem for $X_n$.  
Let $\hat{Y_n}=(Y_n-\lambda n^{1-c})/\sqrt{\lambda n^{1-c}}$. We set 
\[
\gamma_n=\lambda n^{1-c},\quad \delta_n=\frac{\sqrt{\gamma_{n-1}}}{\sqrt{\gamma_n}},
\]
and use the backward difference operator $\nabla \gamma_n=\gamma_n-\gamma_{n-1}$. We obtain 
\begin{equation}
\label{eqn:Yn_DistEqn}
\begin{split}
\hat{Y}_n &\law I_1\cdot  \delta_n\big(\hat{Y}^{\ast}_{n-1}-\nabla \gamma_n\big) + (1-I_1)(1-I_2)\cdot \frac{n-C_{n-1-J_n,J_n} -\gamma_n}{\sqrt{\gamma_n}}\\
&\quad-(1-I_1)I_2\sqrt{\gamma_n} ,
\end{split}
\end{equation}
for $n\ge 2$, with $\hat{Y}_1=-\sqrt{\lambda}$. Let $\phi_n(t)=\E(e^{i t \hat{Y}_n})$ denote the characteristic function of $\hat{Y}_n$
and 
\[
\psi_n(t)=\E(\exp\Big(i t (n-C_{n-1-J_n,J_n}-\gamma_n)/\sqrt{\gamma_n} \Big)).
\]
We obtain the recurrence relation
\begin{equation}
\label{eqn:rec1}
\phi_{n}(t)= p \cdot \phi_{n-1}\Big(\delta_n t\Big)e^{-i t \nabla \gamma_n}+
 (1-p)\big(1-(1-p)^n\big)\psi_n(t)+(1-p)^n\sqrt{\gamma_n},
\end{equation}
for $n\ge 2$, with initial value $\phi_{1}(t)=e^{-it\sqrt{\lambda}}$. Informally, for large $n$ we have
\[
\phi_{n-1}\Big(\delta_n t\Big)e^{-i t \nabla \gamma_n}\to \phi_{\infty}(t),
\]
as well as $\psi_n(t)\to e^{-t^2/2}$, such that
\[
\phi_\infty(t) \approx \lim_{n\to \infty}\big(1-(1-p)^n\big)\psi_n(t).
\]
Of course, this hand-waving argument only motivates the limit law and is highly non-rigorous, as for $p\to 1$ we have $\lim_{n\to\infty}(1-p)=0$. Thus, this argument  actually leads only to a degenerate equation:
\[
\hat{Y}\law \hat{Y}^{\ast},\quad\text{or equivalently }\phi_\infty(t)=\phi_\infty(t).
\]
We sketch how to obtain the result in a rigorous way: first, derive a closed form expression from~\eqref{eqn:rec1}:
\begin{equation*}
\begin{split}
\phi_{n}(t)&=\sum_{k=0}^{n-1}p^k e^{-it\Gamma_{n,k}}\Big((1-p)\big(1-(1-p)^{n-k}\big)\psi_{n-k}(t\fallfak{\delta_n}k)
+(1-p)^{n-k}\fallfak{(\sqrt{\gamma_n})}{k+1}\Big)\\
&\quad+p^{n-1}e^{-it\Gamma_{n,n-1}}\phi_{n}(\fallfak{\delta_n}{n-1}t),
\end{split}
\end{equation*}
where we have used the short-hand notations
\[
\fallfak{\delta_n}k:=\prod_{\ell=0}^{k-1}\delta_{n-k},\quad
\Gamma_{n,k}:=\sum_{\ell=0}^{k-1}\nabla\gamma_{n-\ell}=\gamma_n-\gamma_{n-k}.
\]
Note that the arising sum telescopes, which leads to the stated simplification of $\Gamma_{n,k}$. Of course, we anticipate the most contribution near $k\approx n^{c}t$, $t>0$, from the geometric waiting time, where $\psi_{n-k}(t)$ converges to the characteristic function of the normal distribution. 
There, $\Gamma_{n,k}\approx 1$, as well as $\fallfak{\delta_n}k\approx 1$. As before, we need to approximate the sum by an integral, 
which is then split into two parts, where dominant part is over the interval of the form $[0,n^{c+\epsilon}]$. 
The remaining part $(n^{c+\epsilon},n-1]$ of the integral turns out to be negligible, due to the factor $p^k$ decaying too fast, which leads to the stated limit law. 
\end{proof}

\section{Summary and Outlook}
We consider a card guessing game with complete feedback after an asymmetric riffle shuffle. The optimal strategy for maximizing the number of correct guesses was presented. 
We also obtained several limit laws and phase transitions for the number of correct guesses. Concerning generalizations and extensions for future research, one could try to show moment convergence for $X_n$ by adapting the generating functions approach carried out in~\cite{KuPa2025}.
Moreover, another point of interest is to study the first pure luck case~\cite{KuPa2025}, where we anticipate a phase transition from the arcsine law, observed before. 
Finally, it may be also of interest to study this model for guessing with no feedback at all.

\section*{Declarations of interest and Acknowledgments}
There are no competing financial or personal interests that influenced the work reported in this paper. 

\smallskip 

The author thanks Alois Panholzer for discussions about riffle shuffles and his interest in this work.
\bibliographystyle{cyrbiburl}
\bibliography{CardGuessingUnRiffle-refs}{}


\section*{Appendix}
\begin{code}
For the reader's convenience we include the short \texttt{R}-Code used for obtaining the probabilities in Example~\ref{Ex:3}.
It creates a \text{.csv}-file with the values of $\max_{2\le m}\P\{\FC=m\}$.

\smallskip

\begin{lstlisting}[style=Rstyle]
#Initialising the probabilities
p=0.15; q=1-p

#Determining the end of P(FC=1)>=1/2
n0=floor(log(0.5-p)/log(1-p))

#ranges
r1=(n0+1):40;r2=1:n0;

#ProbabilitiesMode
vecProb=dbinom(floor(r1*p),r1-1,p)*q
modes=c(floor(r1*p))

#ProbabilitiesFC=1
vecOne=c(p+q^r1); 

#ProbabilitiesFC=1, Begin
vecOneStart=c(p+q^r2)

vecProbStart=0;
for (i in 2:n0) 
  {
   upper=i-1;
    v=q*dbinom(1:upper,i-1,p);
    vecProbStart=c(vecProbStart,max(v))  
    }
  
rTotal=c(r2,r1)
vecProb=c(vecProbStart,vecProb)

vecProbPairs<- matrix(c(rbind(rTotal, vecProb)), ncol = 2, byrow = TRUE)
colnames(vecProbPairs) <- c("x","y")
write.csv(vecProbPairs,"unRifflePoints.csv", row.names = FALSE)
\end{lstlisting}
\end{code}

\end{document}